\pdfoutput=1
\documentclass[a4paper,reqno]{amsart}
\usepackage[margin=1.25in]{geometry}

\AtBeginDocument{%
  }

\RequirePackage[
  datamodel=acmdatamodel,
  style=acmnumeric,
  ]{biblatex}

\usepackage{mathtools}
\usepackage{amsthm,thmtools}

\usepackage[american]{babel}

\usepackage{microtype}

\usepackage{enumerate}

\usepackage{balance} 

\usepackage{fontawesome5} 
\usepackage{wasysym} 

\usepackage{tikz-cd}

\usepackage{mathpartir} 

\usepackage{rotating} 
\usepackage{newtxmath} 

\usepackage{hyperref}

\usepackage[capitalize]{cleveref}

\hypersetup{
    colorlinks,
    linkcolor={red!50!black},
    citecolor={blue!50!black},
    urlcolor={blue!80!black}
}

\usetikzlibrary{decorations.markings,decorations.pathreplacing,
  shapes.geometric,matrix,arrows,chains,positioning,scopes}

\pgfdeclarearrow{
  name = pxto,
  setup code = {
    \pgfarrowssettipend{1.5\pgflinewidth}
    \pgfarrowssetbackend{-2.5508\pgflinewidth}
    \pgfarrowssetlineend{-.25\pgflinewidth}
    \pgfarrowssetvisualbackend{-0.021\pgflinewidth}
    \pgfarrowsupperhullpoint{1.5\pgflinewidth}{0\pgflinewidth}
    \pgfarrowsupperhullpoint{-2.0085\pgflinewidth}{3.6525\pgflinewidth}
    \pgfarrowsupperhullpoint{-2.5508\pgflinewidth}{3.0763\pgflinewidth}
  },
  drawing code = {
    \pgfsetdash{}{0pt}%
    \pgfpathmoveto{\pgfpoint{1.5\pgflinewidth}{0.0254\pgflinewidth}}%
    \pgfpathlineto{\pgfpoint{-2.0085\pgflinewidth}{3.6525\pgflinewidth}}%
    \pgfpathlineto{\pgfpoint{-2.5508\pgflinewidth}{3.0763\pgflinewidth}}%
    \pgfpathlineto{\pgfpoint{-0.4322\pgflinewidth}{0.5\pgflinewidth}}%
    \pgfpathlineto{\pgfpoint{-0.4322\pgflinewidth}{-0.5\pgflinewidth}}%
    \pgfpathlineto{\pgfpoint{-2.5508\pgflinewidth}{-3.0763\pgflinewidth}}%
    \pgfpathlineto{\pgfpoint{-2.0085\pgflinewidth}{-3.6525\pgflinewidth}}%
    \pgfpathclose%
    \pgfusepathqfill
  }
}
\tikzset{>=pxto}
\tikzcdset{arrow style=tikz}

\newcommand{\mop}[1]{\mathop{#1}}
\newcommand{\mc}[1]{\mathcal{#1}}

\newcommand{\mbf}[1]{\mathbf{#1}}

\newcommand{\mr}[1]{\mathrm{#1}}
\newcommand{\llb}[1]{\llbracket #1 \rrbracket}

\newcommand{\lift}[2]{\mathop{\langle\uparrow_{#1}^{#2}\rangle}}

\newcommand{\prtt}{\mathrm{T}_{\mathrm{pr}}}

\DeclareMathOperator{\lambdaop}{\lambda}
\DeclareMathOperator{\piop}{\pi}
\DeclareMathOperator{\Gammaop}{\Gamma}
\DeclareMathOperator{\inl}{inl}
\DeclareMathOperator{\inr}{inr}
\DeclareMathOperator{\PSh}{Pr}
\DeclareMathOperator{\Sh}{Sh}
\DeclareMathOperator{\Set}{Set}
\DeclareMathOperator{\yo}{y}

\newcommand{\setN}{\vmathbb{N}}

\newcommand{\synN}{\mathrm{N}}
\DeclareMathOperator{\synz}{zero}
\DeclareMathOperator{\syns}{succ}
\DeclareMathOperator{\synind}{ind}
\newcommand{\pred}{\mathop{\yo_{\mathrm{pred}}}}

\newcommand{\emptyobject}{\vmathbb{0}}
\newcommand{\unit}{\vmathbb{1}}
\newcommand{\synunit}{\mathrm{1}}
\newcommand{\bool}{\vmathbb{2}}

\newcommand{\tm}[1]{\mathop{\mr{tm}_{#1}}}
\newcommand{\tp}[1]{\mathop{\mr{tp}_{#1}}}

\newcommand{\synU}{\mathrm{U}}
\newcommand{\synUfib}{\overline{\mathrm{U}}}

\newcommand{\prubar}{\hat{\mathcal{U}}_0^{\mr{pr}}}
\newcommand{\justV}{\mathcal{U}}
\newcommand{\prv}{\mathcal{U}_0^{\mr{pr}}}
\newcommand{\cprv}{\mathcal{U}_0^{\mr{cpr}}}
\newcommand{\prvbar}{\hat{\mathcal{U}}_0^{\mr{pr}}}

\newcommand{\catR}{\mathrm{R}}
\newcommand{\zeroR}{\mathbf{1}}
\newcommand{\oneR}{\mathbf{N}}
\newcommand{\jfin}{\mathrm{J}_{\mathrm{fin}}}
\newcommand{\bfin}{\mathrm{B}_{\mathrm{fin}}}
\newcommand{\ShR}{\mathop{\mathcal{R}}}
\newcommand{\glue}{\mathop{\mathcal{G}}}
\newcommand{\modal}{\mathop{\fullmoon}}
\newcommand{\notmodal}{\mathop{\newmoon}}

\newcommand{\rotatesim}{\begin{rotate}{90}$\sim$\end{rotate}}

\newcommand{\type}{\:\mathrm{type}}

\DeclareMathOperator{\sgn}{sgn} 

\theoremstyle{theorem}
\declaretheorem[sibling=subsection,style=theorem]{theorem}
\declaretheorem[sibling=theorem,style=theorem]{lemma}
\theoremstyle{remark}
\declaretheorem[sibling=theorem,style=remark]{remark}
\declaretheorem[sibling=theorem,style=remark]{definition}

\newlength{\fulltextwidth}
\begin{document}

\title{Primitive Recursive Dependent Type Theory}

\author[Buchholtz]{Ulrik Buchholtz}
\author[Schipp von Branitz]{Johannes Schipp von Branitz}
\address{University of Nottingham, UK}
\email{{\href{mailto:ulrik.buchholtz@nottingham.ac.uk}{\texttt{ulrik.buchholtz@nottingham.ac.uk}} \\
    {\href{mailto:johannes.schippvonbranitz@nottingham.ac.uk}{\texttt{johannes.schippvonbranitz@nottingham.ac.uk}}}}}
\urladdr{\url{https://ulrikbuchholtz.dk/} \\ \url{https://jsvb.xyz/}}

\begin{abstract}  
  We show that restricting the elimination principle of the natural numbers type in Martin-L{\"o}f Type Theory (MLTT) to a universe of types not containing $\Pi$-types ensures that all definable functions are primitive recursive. This extends the concept of primitive recursiveness to general types. We discuss extensions to univalent type theories and other notions of computability.
  We are inspired by earlier work by Martin Hofmann~\cite{hofmann_application_1997}, work on Joyal's arithmetic universes~\cite{maietti_2010}, and
  Hugo Herbelin and Ludovic Patey's sketched Calculus of Primitive Recursive Constructions~\cite{herbelin-patey-cprc}.
  
  We define a theory $\prtt$ that is a subtheory of MLTT with two universes $\mr{U}_0:\mr{U}_1$, such that all inductive types are finitary and $\mr{U}_0$ is restricted to not
  contain $\Pi$-types:
  \[
  \inferrule{\vdash A:\mr{U}_\alpha \\ a:A\vdash B(a):\mr{U}_\alpha}{\vdash (a:A)\to B(a) : \mr{U}_{\mr{max}(1,\alpha)}}
  \]
  We prove soundness such that all functions $\setN\to\setN$ are primitive recursive. The proof requires that $\prtt$ satisfies canonicity, which we easily prove using synthetic Tait computability~\cite{sterling_first_2021}.
\end{abstract}

\maketitle

\section{Introduction}\label{sec:primrec}
A primitive recursive function is roughly a numerical algorithm that can be computed using only (bounded) $\textrm{for}$-loops. They form a subclass of all computable functions and are commonly used in proofs of relative consistency results.

Martin-L\"{o}f Type Theory (MLTT) is a dependent type theory which can serve as a foundation for mathematics. The variant we consider features $\Sigma$- or dependent pair and $\Pi$- or dependent function types, intensional identity types, basic inductive types and a hierarchy of universes.

The main contribution of this paper is a proof that restricting the elimination principle of the type of natural numbers to a universe not containing $\Pi$-types ensures that all definable terms $n:\synN\vdash f(n):\synN$ are primitive recursive functions under their standard interpretation in the topos $\Set$ of sets. In other words, dependent type theory without $\Pi$-types is a conservative extension of primitive recursive arithmetic (PRA).
The proof proceeds by gluing the set-model to a certain sheaf topos of primitive recursive functions.

PRA is often invoked as a base theory for reverse mathematics~\cite{SimpsonSSOA} and
work in formal metatheory~\cite{KleeneIM}. However, this line of work requires a lot of delicate encoding and would be difficult to mechanize in a proof assistant.
Our work aims to alleviate this problem by giving a subsystem of MLTT (which is itself
the basis for many proof assistants, including Agda, Rocq (née Coq), and Lean),
which is conservative over PRA, but is much more expressive (requiring fewer encodings),
and which is directly amenable for mechanization.
Eventually, we imagine that a tool like Agda could feature a \verb!--pra! flag
to ensure that a file only uses constructions that are conservative over PRA.
Since we provide a modular semantics that ensures this, it is easily possible to extend our syntax with even further conveniences, some of which we discuss in~\Cref{sec:conclusion}.

In \Cref{sec:pr} we recall the basic definitions of primitive recursive functions and their representation in Cartesian closed categories. In \Cref{sec:stc} we fix notation for the tool that is synthetic Tait computability (STC).

This is followed by a sketch of concrete syntax and a formal definition of higher order abstract syntax for our Primitive Recursive Dependent Type Theory (PRTT) in \Cref{sec:syntax}. We discuss examples and applications in \Cref{sec:examples}.

In \Cref{sec:semanticsr} we define a semantics in the topos of sheaves on a category of arities and primitive recursive functions equipped with the finite cover topology. This is followed by a proof that PRTT admits canonical forms in \Cref{sec:canonicity}. The results of those two sections are finally combined in \Cref{sec:soundness} where we define an interpretation of PRTT in a topos, constructed using Artin gluing along the interpretations in the sheaf topos and the standard model. We use the glue topos to show that all PRTT-definable functions are in fact primitive recursive.

A comparison of our results to related work is given in \Cref{sec:related}.

The system PRTT has the downside that while it is complete with respect to primitive recursive functions, not every primitive recursive function can be encoded in a straightforward way. We discuss possible extensions using a comonadic modality, an internal universe of codes for primitive recursive constructions, finitary inductive types and polynomial time computability akin to Hofmann's calculi~\cite{hofmann_application_1997,hofmann_mixed_1998} in \Cref{sec:conclusion}.
We also consider applications to Cubical Type Theory and obstructions to the transfer of our techniques to higher topoi.

\section{Primitive Recursion}\label{sec:pr}
In this section we define primitive recursive functions and explain why one might expect MLTT with natural numbers but without $\Pi$-types to capture exactly primitive recursive functions.
\begin{definition}
	The \emph{basic primitive recursive functions} are constant functions, the successor function and projections out of finite products of $\setN$. A \emph{primitive recursive function} is obtained by finite applications of function composition and the primitive recursion operator
	\begin{align*}
		(l : \setN)\to(\setN^l\to\setN)\to(\setN^{l+2}\to\setN)&\to(\setN^{l+1}\to\setN)\\
		\mr{primrec}^l_{g,h}(0,x)&=g(x)\\
		\mr{primrec}^l_{g,h}(n+1,x)&=h(n,\mr{primrec}^l_{g,h}(n,x),x)
	\end{align*}
\end{definition}
Many functions are primitive recursive. Some examples include addition, exponentiation and the greatest common divisor. One of the simplest and earliest-discovered examples of a total computable function which is not primitive recursive is the Ackermann function~\cite{ackermann_zum_1928}. A simple two-argument variant $\mr{A}$ due to Rózsa Péter is given by
\begin{align*}
	\mr A(0,n)&:=n+1\\
	\mr A(m+1,0)&:=\mr A(m,1)\\
	\mr A(m+1, n+1)&:=\mr A(m,\mr A(m+1, n))
\end{align*}
for non-negative integers $m$ and $n$. One can show that $A$ grows faster than any primitive recursive function and is therefore not primitive recursive. The explosion in growth is caused by the third clause, featuring structural recursion targeting the function type $\setN \to \setN$.

We can transfer the notion of primitive recursiveness to morphisms between natural numbers objects (NNO) in general Cartesian closed categories (CCC) (c.f. \cite{Lambek1986IntroductionTH}).
\begin{definition}
	A function $f:\setN^k\to\setN$ is \emph{representable} in a Cartesian category with weak NNO $\synN$ if there is an arrow $\Tilde{f}:\synN^k\to\synN$ such that
	the following square commutes.
	\[
	\begin{tikzcd}
		\setN^k
		\arrow[r, "f"]
		\arrow[d, swap, "\mr{num}^k"]
		& \setN
		\arrow[d, "\mr{num}"]\\
		\Gammaop \synN^k
		\arrow[r, swap, "\Gammaop \Tilde{f}"]
		& \Gammaop \synN
	\end{tikzcd}
	\]
	Here $\mr{num}(n)$ denotes the $n$-th numeral in $\synN$ and $\Gamma:=\mr{Hom}(\unit,-)$ the global sections functor.
\end{definition}

The following result gives us a hint what a dependent type theory which is sound and complete w.r.t. primitive recursion might look like.
\begin{theorem}[\cite{hofstra_aspects_2020,ROMAN1989267,Lambek1986IntroductionTH}]
	The primitive recursive functions are exactly the representable functions in the free Cartesian category with parametrized NNO. The provably total functions of Peano Arithmetic are exactly the representable functions in the free CCC with weak NNO.
\end{theorem}
On one hand, induction on natural numbers and products suffice to represent all primitive recursive functions. On the other hand, if the category has exponentials, more than just the primitive recursive functions are representable. In other words, when type theory is used as the internal language of such a category, one would expect the presence or absence of $\Pi$-types to determine whether just primitive recursive or all total recursive functions are definable.

\section{Synthetic Tait Computability}\label{sec:stc}
In his PhD Thesis \cite{sterling_first_2021}, Sterling introduces a synthetic method of constructing logical relations in order to prove properties such as normalisation for formal systems. The idea is that, given a left exact functor $\rho:\mr T\to\mathcal S$ from a theory (\emph{qua} syntactic category) to a topos, the Artin gluing $\glue:=\hat\rho\downarrow\mathcal S$ of the extension of $\rho$ along the free cocompletion of $\mr T$ is a topos. This topos comes equipped with an open immersion $j:\PSh\mr T\hookrightarrow\glue$ and a closed immersion $i:\mathcal S\hookrightarrow\glue$ such that $\hat\rho=i^*\circ j_*$. The two immersions act on sheaves as
\begin{align*}
	j_*(X)&=(X,\mr{id}:\hat\rho(X)\to\hat\rho(X))\\
	j^*(X,\varphi:S\to\hat\rho(X))&=X\\
	i_*(S)&=(\unit, !:S\to\hat\rho(\unit))\\
	i^*(X,\varphi:S\to\hat\rho(X))&=S.
\end{align*}
One obtains an open modality $\modal:=j_*\circ j^*$ and a closed modality $\notmodal:=i_*\circ i^*$ on $\glue$. Given the subterminal sheaf
\[\xi:=(\unit, !:\emptyobject\to\hat\rho(\unit)),\]
one sees that $\notmodal (X,\varphi)$ is given by the pushout of the span \[(X,\varphi)\leftarrow \bigl((X,\varphi)\times\xi\bigr)\to\xi.\]
This yields an elimination principle for $\notmodal$-modal types, denoted by a $\mr{try}$-clause.

One uses the internal language of $\mathcal G$ to construct a section $s$ of the projection $j^*$. This section induces squares
\[
\begin{tikzcd}
	S
	\arrow[r, ""]
	\arrow[d, ""]
	& S'
	\arrow[d, ""]\\
	\hat\rho(X)
	\arrow[r, swap, "\hat\rho(f)"]
	& \hat\rho(X')
\end{tikzcd}
\]
for terms $f:X\to X'$ of $\mr T$. The section is constructed by lifting syntactic objects, internally given by $\modal$-modal types, to larger universes, therby attaching computational information to them.

In order to ensure that $s$ is a section up to strict equality rather than isomorphism, one assumes that the relevant universes are \emph{strong}, meaning that they satisfy \emph{realignment} along monomorphisms. In our exposition we leave realignment mostly implicit. It should be remarked that the existence of strong universes is not necessarily constructive in general, and one way to approach a soundness proof without realignment would be to try to use the techniques presented in \cite{bocquet_et_al:LIPIcs.FSCD.2023.18}.

We adopt Sterling's notation
\[\{A\mid\xi\hookrightarrow a\}\]
for \emph{extent types}, that is types whose elements are strictly aligned over some $\xi$-partial element $a$.

If $\rho$ is the global sections functor, then one can choose a section which aligns canonical forms over closed terms and hence proves canonicity for $\mr T$. We shall employ this technique to prove canonicity for PRTT in \Cref{sec:canonicity}.

Sterling's logical framework is a dependent type theory which produces an LCCC $\mr T$ of judgements given its specification as a total space over the judgement classifier. Such a specification is typically given using Agda-style $\textrm{record}$ syntax. A complete syntax of the abstract syntax for our theory is presented in \Cref{fig:PRTT}.

Furthermore, we construct a model $\llb{-}_{\ShR}:\prtt\to\ShR$ in a certain topos of sheaves on a site of primitive recursive functions and form the glue topos along the functor
\begin{align*}
	\rho:\PSh\prtt&\to\Set\\
	X&\mapsto\Gamma(\widehat{\llb{X}_{\ShR}})\times\widehat{\llb{X}_{\Set}}
\end{align*}
(with $\widehat{\llb{-}_{\ShR}}$ and $\widehat{\llb{-}_{\Set}}$ denoting the corresponding Yoneda extensions),
in order to prove that the standard interpretation of a term $n:\synN\vdash f(n):\synN$ in $\Set$ is indeed primitive recursive.

\section{Syntax of Primitive Recursive Dependent Type Theory}\label{sec:syntax}
We would like to restrict the induction principle of an inductive natural numbers type without entirely removing dependent products from the theory. To do so we assume a universe $\synU_0$ containing basic inductive types, which is closed under $\Sigma$-types and identity types, but not function types. The elimination principle of the natural numbers $\synN$ is restricted to type families in $\synU_0$.

An illustrative portion of the concrete syntax for PRTT is presented in~\cref{fig:PRTTRules}. Later we shall only work with its formal account as presented in~\cref{fig:PRTT}.
\begin{figure*}
	\caption{Rules of $\prtt$.}
	\label{fig:PRTTRules}
	\raggedright Rules which differ from standard MLTT. All judgements $\Delta\vdash\mathcal{J}$ have an implicit prepended context $\Gamma$.
	\begin{minipage}{\fulltextwidth}
    \footnotesize
		\begin{align*}
			&\mprset{flushleft}\inferrule{\:}{\vdash\mr{U}_\alpha\type}
			\mprset{flushleft}\inferrule{\vdash A:\mr{U}_\alpha}{\vdash A:\mr{U}_{\alpha+1}}
			\mprset{flushleft}\inferrule{a:A\vdash B(a)\type}{\vdash (a:A)\to B(a)\type}
			\mprset{flushleft}\inferrule{\vdash A:\mr{U}_\alpha \\ a:A\vdash B(a):\mr{U}_\alpha}{\vdash (a:A)\to B(a) : \mr{U}_{\mr{max}(1,\alpha)}}\\
			&\mprset{flushleft}\inferrule{a:A\vdash b(a):B(a)}{\vdash\lambdaop (a:A) . b(a) : (a:A)\to B(a) \\\\ x:A\vdash (\lambdaop (a:A).b(a))(x)=b(x):B(x)}
			\mprset{flushleft}\inferrule{\vdash f:(a:A)\to B(a) }{a:A\vdash f(a) : B(a) \\\\ \vdash\lambdaop (a:A).f(a)=f : (a:A)\to B(a)}\\
			&\mprset{flushleft}\inferrule
			{\:}
			{\vdash\synN\type \\ \vdash \synN : \synU_0 \\\\ \vdash\synz:\synN \\ n:\synN\vdash\syns n:\synN}
			\mprset{flushleft}\inferrule
			{n:\synN\vdash C(n) : \synU_0 \\\\ \vdash c_{\synz} : C(\synz) \\\\ n : \synN, c :C(n)\vdash c_{\syns}(n,c) : C(\syns n)}
			{n:\synN\vdash\synind_{c_{\syns}, c_{\synz}}(n) : C(n) \\\\ \vdash\synind_{c_{\syns}, c_{\synz}}(\synz)=c_{\synz} : C(\synz) \\\\
				n:\synN\vdash\synind_{c_{\syns}, c_{\synz}}(\syns n)=c_{\syns}(n,\synind_{c_{\syns}, c_{\synz}}(n)) : C(\syns n)}
		\end{align*}
	\end{minipage}
\end{figure*}
\begin{figure*}
	\caption{Higher Order Abstract Syntax for $\prtt$.}
	\label{fig:PRTT}
  \raggedright Curly braces denote implicit arguments. The judgement levels $\alpha\leq\beta\leq\gamma$ range over a finite or countable linear order $\{0,\ldots\}$.
	\begin{minipage}{\fulltextwidth}
    \footnotesize
		\begin{align*}
			\mathrm{record}&\:\prtt : \mr{SIG} \:\mathrm{where}\\
			\tp\alpha&:\Box\\
			\tm\alpha&: \tp\alpha\to\Box\\
			\lift{\alpha}{\beta} &: \tp\alpha\to\tp\beta\\
			\_ &: \{A\}\to(\lift{\alpha}{\alpha} A =_{\tp\alpha} A)\\
			\_ &:\{A\}\to(\lift{\beta}{\gamma}\lift{\alpha}{\beta}A=_{\tp\gamma}\lift{\alpha}{\gamma}A)\\
			\_ &: \{A\}\to(\tm\alpha A=_\Box\tm\beta(\lift{\alpha}{\beta}A))\\
			\_ &:\{A,B\}\to(\lift{\alpha}{\beta}A=_{\tp\beta}\lift{\alpha}{\beta}B)\to (A=_{\tp\alpha} B)\\
			\Sigma_\alpha &: (A:\tp\alpha)\to(B:\tm\alpha A\to\tp\alpha)\to\tp\alpha\\
			\mr{pair}_\alpha &: \{A,B\}\to ((a : \tm\alpha A)\times\tm\alpha(\mathop B a))\cong\tm\alpha(\Sigma_\alpha(A,B))\\
			\_ &:\{A,B\}\to\lift{\alpha}{\beta}\Sigma_\alpha(A,B)=_{\tp\beta}\Sigma_\beta(\lift{\alpha}{\beta} A,\lift{\alpha}{\beta}\circ B)\\
			\mr{eq}_\alpha &: (A : \tp\alpha)\to(a,b:\tm\alpha A)\to\tp\alpha\\
			\mr{refl}_\alpha &: \{A\}\to(a : \tm\alpha A)\to\mr{eq}_\alpha(a,a)\\
			\mr{eqind}_\alpha&:
			\{A,a\}
			\to\bigl(P: (b:\tm\alpha A)\to (p:\mr{eq}_\alpha(a,b))\to\tp\alpha)\bigr)
			\to\{b,p\} 
			\to\tm\alpha(P(a,\mr{refl}(a)))
			\to\tm\alpha(P(b,p))\\
			\_&:\{A,a,P\}
			\to\bigl(s:\tm\alpha(P(a,\mr{refl}_\alpha(a)))\bigr)
			\to \bigl(s=_{\tm\alpha(P(a,\mr{refl}_\alpha(a)))}\mr{eqind}_\alpha(P,\{a\},\{\mr{refl}_\alpha(a)\},s)\bigr)\\
			\_ &:\{A\}\to\lift{\alpha}{\beta}\mr{eq}_\alpha(A)=_{\tp\beta}\mr{eq}_\beta(\lift{\alpha}{\beta}A)\\
			\Pi_\alpha &: (A:\tp\alpha)\to(B:\tm\alpha A\to\tp\alpha)\to\tp{\mr{max}(1,\alpha)}\\
			\lambdaop_\alpha &: 
			\{A,B\}
			\to ((x : \tm{\mr{max}(1,\alpha)}(A))\to\tm{\mr{max}(1,\alpha)}(B(x)))\cong\tm{\mr{max}(1,\alpha)}(\Pi(A,B))\\
			\_&:\{A,B\}\to\lift{\alpha}{\beta}\Pi_{\alpha}(A,B)=_{\beta}\Pi_{\beta}(\lift{\alpha}{\beta} A,\lift{\alpha}{\beta}\circ B)
			\quad (\text{where}\:\alpha>0)\\
			\emptyset &:\tp 0\\
			\mr{exfalso} &: (\mop B:\tm 0\emptyset\to\tp 0)\to(x:\tm 0\emptyset)\to\tm 0(\mop B x)\\
			\synunit &:\tp 0\\
			\star &: \tm 0\synunit\\
			\mr{unitind} &: 
			(\mop B:\tm 0\synunit\to\tp 0)
			\to(b:\tm 0(\mop B\star))
			\to(x:\tm 0\synunit)
			\to\tm 0(\mop B x)\\
			\_ &: \{B,b\}\to(\mr{unitind}(B,b,\star)=_{\tm 0(\mop B\star)}b)\\
			\synN&: \tp 0\\
			\synz&: \tm 0\synN\\
			\syns&: \tm 0\synN\to\tm 0\synN\\
			\synind&:(\mop B: \tm 0\synN\to\tp 0)\\
			&\quad\to(b_{\synz}:\tm 0(\mop B\synz))\\
			&\quad\to(b_{\syns}:(n : \tm 0\synN)\to\tm 0(\mop B n)\to \tm 0(\mop B(\syns n)))\\
			&\quad\to(x : \tm 0\synN)\to\tm 0(\mop B x)\\
			\_&:\{B,b_{\synz},b_{\syns}\}\to(\synind(B,b_{\synz},b_{\syns},\synz)=_{\tm 0(\mop B\synz)}b_{\synz})\\
			\_&:\{B,b_{\synz},b_{\syns},n\}\to\bigl(\synind(B,b_{\synz},b_{\syns},\syns n)=_{\tm 0(B(\syns n))} \mop{b_{\syns}}(\synind(B,b_{\synz},b_{\syns},n))\bigr)\\
			\synU_\alpha &: \tp{\alpha+1}\\
			\_ &:(\tm{\alpha+1}\synU_\alpha=_\Box\tp\alpha)
		\end{align*}
	\end{minipage}
\end{figure*}

A detailed account of the typical rules of MLTT is given in~\cite{rijke2022introduction}. The only notable differences between that system and ours are that $\Pi$-formation of a family of types of universe level $0$ raises the level by one, and that the induction principle for the natural numbers produces an open term $n:\synN\vdash\synind(n):B(n)$ rather than a function of type $(n:\synN)\to B(n)$. So, in addition to the rules in~\Cref{fig:PRTTRules}, we assume the standard rules for a reflexive, symmetric and transitive judgemental equality relation and variable substitution. Additionally, we assume $\Sigma$- and identity type closure for every universe and closure under $\Pi$-types for universes $\synU_\alpha$ with $\alpha>0$. Note that we do not assume univalence nor any extensionality principle for propositional equality. The universe levels $\alpha\leq\beta\leq\gamma$ belong to a finite or countable linear order $\{0<1<\ldots\}$. We denote meta-level $\Sigma$- or dependent pair types by $(a:A)\times B(a)$ and $\Pi$- or dependent function types by $(a:A)\to B(a)$. We make no distinction between a type $A$ and its code $A:\synU$ in a universe, although this is of course present in the formal account.

The logical framework converts a theory presented as a nested $\Sigma$-type in Agda-style record notation into a locally Cartesian closed (LCC) category of judgements $\prtt$. For example, $\Box$ denotes a judgement classifier and $\tm\alpha:\tp\alpha\to\Box$ introduces for every type judgement $A:\tp\alpha$ a judgement for its terms. A model of $\prtt$ is an LCC functor $\llb{-}_{\mr C}:\prtt\to\mr C$.

One part of the adequacy statement for $\prtt$ is the next theorem. Its counterpart is \Cref{thm:prttsound}.

\begin{theorem}
	The theory $\prtt$ is complete with respect to primitive recursion in the sense that any primitive recursive function can be defined in it.
\end{theorem}
\begin{proof}
	The basic primitive recursive functions are clearly definable using $\synind$. Function composition is given by substitution. The term $\synind$ also covers the recursor $\mr{primrec}^l$. The extra variables are absorbed by the context.
\end{proof}

\begin{remark}
	The theory $\prtt$ is a subtheory of MLTT. As such, any model for MLTT is a model for $\prtt$. This includes any presheaf topos on a small category. In particular, we have a standard model
	\[\llb{-}_{\Set}:\prtt\to\Set\]
	in the topos of sets. This model maps the syntactic natural numbers $\synN$ to the actual natural numbers $\setN$.
	Below, we argue that any definable function \[\llb{f}_{\Set}:\Set(\llb{\synN}_{\Set},\llb{\synN}_{\Set})\]
	is actually primitive recursive by gluing along an interpretation of $\prtt$ in a topos of primitive recursive functions.
\end{remark}

\begin{remark}
	The calculus $\prtt$ can be amended with a natural numbers type $\vdash\synN_\alpha:\tp\alpha$ for $\alpha>0$ which has an elimination principle for the larger universes also closed under $\Pi$-types and which is mapped to the natural numbers object in the interpretations above. Initiality gives a coercion $\synN_\alpha\to\lift{0}{\alpha}\synN$.
\end{remark}

\section{Examples and Applications}\label{sec:examples}

To illustrate how to use $\prtt$, consider the Cantor Normal Forms for ordinals
less than $\varepsilon$, for instance as developed in~\cite{KNFX}.
First we can define a (primitive recursive) isomorphism $\synN \simeq \synunit + \synN \times \synN$, which allows us to think of numbers as unlabeled binary trees,
with a corresponding induction principle. We write $0$ for the element in the
left component
and $\omega^\alpha + \beta$ for elements in the right component.
By recursion, we define the ordering relation $<$ (with values in the booleans $\synunit + \synunit$) such that
\begin{align*}
  0 &< \omega^a+b \\
  \alpha < \gamma \to{} \omega^\alpha +\beta{} &< \omega^\gamma + \delta \\
  \beta < \delta \to{} \omega^\alpha +\beta{} &< \omega^\alpha + \delta,
\end{align*}
and a decidable predicate $\mathrm{isCNF}$ such that
\begin{align*}
  &\mathrm{isCNF}(0) \\
  \mathrm{isCNF}(\alpha) \to \mathrm{isCNF}(\beta) \to \mathrm{left}(\beta)\le \alpha
  \to{}&\mathrm{isCNF}(\omega^\alpha+\beta),
\end{align*}
where the function $\mathrm{left}$ gives the left subtree if it exists, else $0$,
and $\alpha\le\beta :\equiv (\alpha<\beta + \alpha=\beta)$.
Then $\prtt$ can prove $\mathrm{CNF} :\equiv \sum_{\alpha:\synN}\mathrm{isCNF}(\alpha)$
is totally ordered. Of course, by our soundness result, $\prtt$ cannot prove induction along $<$ on $\mathrm{CNF}$, as this would amount to induction up to $\varepsilon_0$.
But using the universe $\synU_1$, this can at least be stated. If we encode the syntax of arithmetic and a proof calculus for classical first order logic, we can then prove in $\prtt$ the consistency of Peano arithmetic (PA) assuming this induction principle.

Likewise, by encoding other finitary inductive type families, it is possible to encode more complicated ordinal notation systems and the syntax of more complicated foundational systems, such as type theory itself. It is then possible to define translation functions, for instance the double-negation translation from PA to Heyting's arithmetic (HA), forcing translations, realizability translations, etc.
This would be even easier in an extension of $\prtt$ with built-in support for such finitary inductive type families. We discuss this extension below in~\cref{sec:conclusion}.

\section{Semantics in a Topos of Primitive Recursive Functions}\label{sec:semanticsr}
This section begins with the construction  of a certain sheaf topos $\ShR$ of primitive recursive functions. The remainder of this section is dedicated to the proof of the following result.
\begin{theorem}\label{thm:interpretationR}
	There is a sound interpretation
	\[\llb{-}_{\ShR}:\prtt\to\ShR\]
	satisfying
	\[\llb{\synN}_{\ShR}=\yo_\oneR. \]
\end{theorem}
Here $\yo$ is the Yoneda embedding, and $\oneR$ is a designated object in the site,
as explained below.
The topos $\ShR$ has the property that morphisms \[\ShR(\llb{\synN}_{\ShR},\llb{\synN}_{\ShR})\]
are exactly primitive recursive functions $\setN\to\setN$, which forms a cornerstone of our gluing argument. The sheaf $\llb{\synN}_{\ShR}$ is a (non-initial) natural numbers algebra $\yo_\oneR$ and $\llb{\synU_0}_{\ShR}$ is described as a type $\prvbar$ of types $X$ with a retraction $r:\yo_\oneR\to X+\unit$. To show that $\prvbar$ contains the basic inductive types we need show that the sheafification of the presheaf $\bool$ is a retract of $\yo_\oneR$. We prove soundness of the elimination principle of $\synN$ into $X$ by encoding terms $x:X$ as numbers $s(x):\yo_\oneR$ via the section $s$ of $r$. The interpretation of the other type formers and higher universes is standard.

In this section we fix a primitive recursive bijection with primitive recursive inverse of type $\setN\to\setN\times\setN$.

\subsection{A Sheaf Topos of Primitive Recursive Functions}
We begin with the definition of a Grothendieck topology on a category $\catR$ of arities and primitive recursive functions.
\begin{definition}
	The category $\catR$ has two objects $\zeroR$ and $\oneR$, related by morphisms $\catR(\oneR, \oneR)$, the primitive recursive functions of type $\setN\to\setN$, and $\catR(\zeroR, \oneR):=\setN$. The object $\zeroR$ is terminal. We sometimes write $\oneR^0$ instead of $\zeroR$.
\end{definition}

\begin{lemma}\label{lem:rfinprod}
	The category $\catR$ has finite products. It does not have all finite limits.
\end{lemma}
\begin{proof}
	We have that $\oneR\times\oneR\cong\oneR$ in $\catR$.
	However, the cospan
	\[\lambdaop \_ . 0 : \oneR\to\oneR\leftarrow\oneR : \lambdaop \_ . 1\]
	cannot be completed to a square in $\catR$.
\end{proof}

\begin{remark}\label{rem:equivr}
	The category $\catR$ is equivalent to the category $\catR'$ with countably many objects $\oneR^l$, and morphisms $\catR'(m,n)$ primitive recursive functions $\setN^m\to\setN^n$.
  This is the Lawvere theory for primitive recursion. It is analogous to the category $\mathbb P$ of arities and polytime functions in~\cite{hofmann_application_1997}, where it is used to proof soundness of a caluculus of polytime functions.
\end{remark}

We define a Grothendieck topology on $\catR$ because the topos $\PSh\catR$ of presheaves on $\catR$ does not have the desired structure, see \Cref{rem:notpresheaves}.

\begin{lemma}
	For primitive recursive $f:\setN^n\to \setN$ and $g:\setN^m\to\setN$, the set
	\[\{ f=g\}:=\{ (x,y)\in\setN^n\times\setN^m\mid fx=gy\}\]
	is decidable. This means that there is a primitive recursive function
	\[\chi_{\{f=g\}}:\setN^{n+m}\to\setN\]
	such that $\chi_{\{f=g\}}(x)=1\Leftrightarrow x\in \{f=g\}$ and $\chi_{\{f=g\}}(x)=0\Leftrightarrow x\not\in \{f=g\}$.
	Furthermore, complements and intersections of decidable subsets are decidable.
\end{lemma}
\begin{proof}
	Use
	\[\chi_{\{ f=g\}}(x,y)=1-\sgn(\mr{max}(fx,gy)-\mr{min}(fx,gy)),\]
	where $\sgn 0=0$ and $\sgn(n+1)=1$.
\end{proof}

\begin{lemma}\label{lem:jfinbasis}
	Let $\jfin$ be the Grothendieck topology on $\catR$ generated by the basis $\bfin$ consisting of finite jointly surjective families. This means that a family of morphisms $\{f_i:\oneR^{n_i}\to\oneR\}$ with  $n_i\in\{0,1\}$ is basic iff it is finite and the induced map out of the coproduct in $\Set$ is surjective.
\end{lemma}
\begin{proof}
	We show that $\jfin$ is indeed a basis. It is clear that isomorphisms define basic covers and that families of composites of basic covering families are basic covers as well. We need to check that if $\{ f_i\}\in \bfin(\oneR)$, then for any $g:\catR(\oneR,\oneR)$ there exists a basic cover $\{ h_j\}\in \bfin(\oneR)$ such that for each $j$, the morphism $g\circ h_j$ factors through some $f_i$. Covers of $\zeroR$ are trivial and singletons $\zeroR\to\oneR$ can be replaced by constant maps $\oneR\to\oneR$. 
	Let 
	\[S_i:=\{f_i=g\land f_1\neq g\land\cdots\land f_{i-1}\neq g\}.\]
	Then the $S_i$ cover $\oneR$ and are pairwise disjoint. For every finite, non-empty $S_i$ define finitely many $h_{i,j}:\zeroR\to\oneR$ picking out the elements of $S_i$. For every infinite $S_k$ define an enumeration $h'_k:\oneR\to\oneR$ factoring through $S_k$. Clearly, for each $i,k$ and $j$ we have $f_i=g\circ h_{i,j}$ and $f_k=g\circ h'_k$, and the $h_{i,j}$ and $h'_k$ together cover $\oneR$.
\end{proof}
We denote the sheaf topos as follows:
\[\ShR:=\Sh(\catR,\jfin)\]
The topology $\jfin$ is subcanonical, so $\yo_\oneR$ is a sheaf and
\[\ShR(\yo_\oneR,\yo_\oneR)\cong\catR(\oneR,\oneR).\]

Just like any Grothendieck topos, $\ShR$ has a natural numbers object $\setN$.
\begin{theorem}
	The representable sheaf $\yo_\oneR$ is not a natural numbers object in $\PSh\catR$ nor in $\ShR$.
\end{theorem}
\begin{proof}
	There are two ways to see this. One is that, if $\yo_\oneR\cong\setN$ were true, then the Ackermann function would be representable and hence primitive recursive.
	
	Alternatively, assume there were a natural transformation $\alpha$ making the square
	\[
	\begin{tikzcd}
		\yo_\oneR
		\arrow[r, "\syns"]
		\arrow[d, swap, "\alpha"]
		& \yo_\oneR
		\arrow[d, "\alpha"]\\
		\setN
		\arrow[r, swap, "+1"]
		& \setN
	\end{tikzcd}
	\]
	commute. By the Yoneda lemma, there is a number $n$ such that at component $k$, $\alpha$ acts as
	\[\alpha(k)(f)=\setN (f)(n)=\mr{id}_{\setN}(n)=n,\]
	where $f:\setN^k\to\setN$ is primitive recursive. Chasing $k:\yo_\oneR$ along both legs of the square above, we get the contradiction
	\[\alpha(k)+1=n+1\neq n=\alpha(\syns (k)).\qedhere\]
\end{proof}

\subsection{Universes}
Let us have a look at some candidates for $\llb{\synU_0}_{\ShR}$.
\begin{definition}\label{def:pruniverses}
	We assume strong cumulative universes $\justV_0<\justV_1$ in $\ShR$ and define the universes $\prv,\prvbar$ and $\cprv$. We have the types (as defined in the internal language)
	\begin{align*}
		\prv:=&(X:\justV_0)\\
		&\times\Bigl((g:X)\\
		&\quad\to(h:\yo_\oneR\to X\to X)\\
		&\quad\to\bigl((f:\yo_\oneR\to X)\times\mr{comp}(f,g,h)\bigr)\Bigr)
	\end{align*}
	and
	\begin{align*}
		\cprv:=(X:\justV_0)\times\Bigl((&\Gamma : \justV_0)\\
		&\to(g:\Gamma\to X)\\
		&\to(h:\Gamma\to\yo_\oneR\to X\to X)\\
		&\to\bigl((f:\Gamma\to\yo_\oneR\to X)\\
		&\qquad\times((\gamma : \Gamma)\to\mr{comp}(f^\gamma,g^\gamma, h^\gamma))\bigr)\Bigr)
	\end{align*}
	with
	\begin{align*}
		\mr{comp}(f,g,h):=&(f(\synz)=g)\\
		&\times\bigl((n:\yo_\oneR)\to(f(\syns n)=h(n,fn))\bigr).
	\end{align*}	
	We also have the object
	\begin{align*}
		\prvbar:=&(X:\justV_0)\times (r : \yo_\oneR\to (X+\unit))\\
		&\times(s : (X+\unit)\to\yo_\oneR)\times (r\circ s = \mr{id}_{X+\unit})
	\end{align*}
	of types $X$ together with a retraction $\yo_\oneR\to X+\unit$.
\end{definition}
The universe $\prv$ consists of the $\justV_0$-small types which $\synN$ can eliminate into. The variant $\cprv$ captures $\yo_\oneR$ as a parametrised natural numbers algebra by introducing an arbitrary context $\Gamma$.
Ideally, we would like to put $\llb{\synU_0}_{\ShR}:=\prv$,
but we were unable to prove that $\prv$ is closed under $\Sigma$-types. The same problem holds for $\cprv$. Instead, we define
\[\llb{\synU_0}_{\ShR}:=\prvbar\]
and use coercions
\[\prvbar\to\prv\leftrightarrow\cprv\]
to show that functions $\yo_\oneR\to X$ with $X:\prvbar$ can be defined by induction.

\begin{lemma}\label{lem:praddcontext}
	There is a retraction $\phi:\cprv\to\prv$ forgetting the context $\Gamma$ and preserving the base type $X$.
\end{lemma}
\begin{proof}
	Correctness of $\phi$ is trivial. To see that $\phi$ has a section, assume an elimination principle without context into $X$, together with terms $g:\Gamma\to X$ and $h:\Gamma\to\yo_\oneR\to X\to X$. Then we can absorb $\Gamma$ into the context; for every $\gamma:\Gamma$ we get a map $f^\gamma:N\to X$, i.e., an appropriate term of type $\Gamma\to\yo_\oneR\to X$.
\end{proof}

\subsection{Finite Types in the Universes}
We show that $\prvbar$ contains the sheafifications of the constant presheaves $\emptyobject$ and $\unit$. The section $s$ allows us to encode elements of $X$ as elements of $\yo_\oneR$,
while the retraction is the corresponding decoding.
The summand $\unit$ serves to allow $X$ to be empty.
\begin{lemma}
	$\emptyobject:\prvbar$.
\end{lemma}
\begin{proof}
	There is exactly one map $r_\emptyobject:\yo_\oneR\to(\emptyobject+\unit)$ given by $n\mapsto\inr\star$. Any element of $\setN$ yields a section via the Yoneda embedding.
\end{proof}

\begin{lemma}\label{lem:sheaf2retract}
	 The sheafification $\bool^+$ of $\bool$ is a retract of $\yo_\oneR$ in $\ShR$. In other words, $\unit:\prvbar$.
\end{lemma}
\begin{proof}
	We first remark that $\bool$ is separated, because no object of the site $(\catR,\jfin)$ is covered by the empty family. Therefore, the $+$-construction only needs to be applied once.
	
	An element of $\bool^+(\oneR)$ is an equivalence class of matching families
	\[\{x_f\in\bool\mid f:\oneR\to\oneR\in P\}\]
	for some cover $P\in \jfin(\oneR)$ satisfying
	\[x_{f\circ g}=x_{f}\]
	for all $g:\oneR^l\to\oneR$. Here two such matching families $\{x_f\mid f\in P\}$ and $\{y_g\mid g\in Q\}$
	are equivalent when there is a common refinement $T\subset P\cap Q$ with $T\in \jfin(\oneR)$ such that $x_f=y_f$ for all $f\in T$.
	By a similar construction as the one used to obtain the disjoint $S_i$ in the proof of~\Cref{lem:jfinbasis}, each matching family is completely determined by finitely many
	\[\{x_{g_j}\}_{j=1}^m\]
	where $\{g_j\}$ is a finite, jointly surjective, \emph{disjoint} family of morphisms with domain and codomain $\oneR$. In fact, we can collect all of the $g_j$ where $x_{g_j}=0$, and similarly for $x_{g_j}=1$ to obtain an equivalent matching family $\{0,1\}$ on two morphisms.\footnote{That is true if not all sections are $0$, or $1$. Otherwise, we get a singleton matching family $\{0\}$ or $\{1\}$.}
	
	The action of $\bool^+$ on a morphism $\varphi:\oneR\to\oneR$ restricts a matching family $\{x_f\}$ to its pullback along $\varphi$. This restriction does not change the value of any individual $x_f$.
	
	By the Yoneda lemma,
	\[\ShR(\yo_\oneR,\bool^+)\cong\PSh\catR(\yo_\oneR,\bool^+)\cong\bool^+(\zeroR).\]
	Under above isomorphism, an element $u\in\bool^+(\zeroR)$ is sent to the natural transformation, here denoted $r_u$, which at component $m\in\{0,1\}$ is given by
	\begin{align*}
		r_u(m):(\oneR^m\to\oneR)&\to \bool^+(m)\\
		\varphi&\mapsto\bool^+(\varphi)(u).
	\end{align*}
	We need to find an equivalence class of matching families $u$, and a natural transformation $s$ with components
	\[s_m:\bool^+(m)\to(\oneR^m\to\oneR)\]
	such that for all $x\in\bool^+(m)$,
	\begin{equation*}
		\bool^+(s_m(x))(u)=x,
	\end{equation*}
	natural in $m$.
	Because sheafification is a reflective localization, $s$ is completely determined by a map $\Tilde{s}:\bool\to\yo_\oneR$, and $\Tilde{s}=s\circ\eta_{\bool}$. By the universal property of $\bool=\unit+_{\PSh\catR}\unit$, $s$ is determined by two global sections of $\yo_\oneR$. Let us choose the constant functions $c_0$ and $c_1$.

	For our retraction we use the cover of $\oneR$ generated by
	$\_\cdot 2, \_\cdot 2+1:\oneR\to\oneR$ and the matching family $u$ with value $\mr a\in\bool$ on $\_\cdot 2$ and $\mr b$ on the other map.
	If we denote the two global sections of $\bool$ by $\mbf a$ and $\mbf b$, we need to show
	\begin{align*}\label{eq:twosheafa}
		\bool^+(c_0)(u)=\mbf a\quad\quad\text{and}\quad\quad \bool^+(c_1)(u)=\mbf b.
	\end{align*}
	It is easy to see that the pullback $c_n^*(S)$ of any $\jfin$-sieve $S$ on $\oneR$ along a constant function $c_n:\oneR\to\oneR$ is the maximal sieve on $\oneR$. Since the restriction maps $\bool^+(f)$ are locally (on each component of the equivalence class of) the matching family given by identity functions, it follows that the desired equations hold.
\end{proof}

\begin{remark}\label{rem:notpresheaves}
	The reason we work in the topos $\ShR$ rather than $\PSh\catR$ is that $\bool$ is not a retract of $\yo_\oneR$ in $\PSh\catR$, because retracts of representables in presheaf categories are tiny, but the endofunctor $X\mapsto X\times X$ on $\PSh\catR$ does not preserve colimits.
\end{remark}

\subsection{\texorpdfstring{$\yo_\oneR$}{yN}-elimination}
In this subsection we show that $\yo_\oneR$-retracts satisfy the elimination principle $\synind$ in several steps.

\begin{lemma}\label{thm:pruyo1}
	The representable sheaf $\yo_\oneR$ is in $\prv$.
\end{lemma}
\begin{proof}
	We show the result for the equivalent category $\catR'$ from \Cref{rem:equivr}.
	Note that for any $F:\PSh\catR$ and $l\in\setN$,
	\[(\yo_\oneR\Rightarrow F)(\oneR^l)\cong\PSh\catR(\yo_\oneR\times\yo_{\oneR^l},F)\cong F(\oneR^{l+1}).\]
	We define the presheaf $F^+$ by $F^+(\oneR^l):=F(\oneR^{l+1})$. It follows that
	\begin{align*}
		&\PSh\catR(\yo_\zeroR,\yo_\oneR\to(\yo_\oneR\to\yo_\oneR\to\yo_\oneR)\to\yo_\oneR\to\yo_\oneR)\\
		\cong&\PSh\catR(\yo_\oneR\times\yo_\oneR^{++}\times\yo_\oneR,\yo_\oneR)
	\end{align*}
	We can define such a natural transformation which at component $l$ is of type
	\begin{align*}
		(\oneR^l\to \oneR)\times(\oneR^{l+2}\to\oneR)\times(\oneR^l\to\oneR)\to(\oneR^l\to\oneR)
	\end{align*}
	and given by
	\[(g,h,n)\mapsto \lambdaop x.\mr{primrec}_{g,h}^l(n(x),x).\]
	The computation rules are easy to verify.
\end{proof}

\begin{lemma}\label{lem:vcindbool}
	The sheaf $\bool^+$ satisfies the contextual elimination principle, i.e., $\bool^+:\cprv$.
\end{lemma}
\begin{proof}
	By \Cref{lem:praddcontext} it suffices to show that $\bool^+:\prv$. Let $(\bool^+,r,s,p)$ be a retraction as constructed in \Cref{lem:sheaf2retract}, $g:\bool^+$, and $h:\yo_\oneR\to\bool^+\to\bool^+$.
	We define
	\[\Tilde g:=s(g):\yo_\oneR\]
	and
	\begin{align*}
		\Tilde h:\yo_\oneR\to\yo_\oneR&\to\yo_\oneR\\
		\Tilde h(n,x)&:=sh(n,rx).
	\end{align*}
	By $\yo_\oneR$-induction (c.f. \Cref{thm:pruyo1}) we get
	\begin{align*}
		\Tilde f:\yo_\oneR&\to\yo_\oneR\\
		\Tilde f(0)&=s(g)\\
		\Tilde f(\syns (n))&=sh(n,r\Tilde fx).
	\end{align*}
	We define the desired morphism $f:=r\circ\Tilde f$.
	Then
	\begin{align*}
		f(0)=rsg=g
	\end{align*}
	and
	\begin{align*}
		f(\syns n)=rsh(n,r\Tilde fn)=h(n,fn)
	\end{align*}
	for all $n:\yo_\oneR$.
\end{proof}
\begin{remark}
	The construction of \Cref{lem:vcindbool} works for any type $X$ with retraction $\yo_\oneR\to X$. The difficult part is proving that a retraction $\yo_\oneR\to X+\unit$ gives $X:\prv$, and contextual elimination into $\bool^+$ suffices to prove that, see below.
\end{remark}

\begin{lemma}\label{lem:predsec}
	The predecessor function $\pred$ is a section of the successor function viewed as
	\[\syns:\yo_\oneR\to(n:\yo_\oneR)\times (n\neq 0).\]
	This is true in $\PSh\catR$ and $\ShR$.
\end{lemma}
\begin{proof}
	We prove the result for $\PSh\catR$ first. We cannot yet use induction on $\yo_\oneR$. Let $n:Z\to\yo_\oneR$ such that $n\neq 0$. We show that $\syns(\pred n)=n$ so that function extensionality yields the result. We may assume that $Z$ is representable (cf. \cite[III.6 and VI.7]{MacLane1994}). Then the result follows immediately from the Yoneda lemma.

	Because the $\jfin$ is subcanonical, application of sheafification to above identity implies the result for $\ShR$.
\end{proof}

\begin{lemma}
	$\yo_\oneR:\prvbar$
\end{lemma}
\begin{proof}
	We inductively define
	\begin{align*}
		s:\yo_\oneR+\unit&\to\yo_\oneR\\
		n&\mapsto\syns n\\
		\star&\mapsto 0.
	\end{align*} 
	By case distinction we get a map
	\begin{align*}
		r:\yo_\oneR&\to\yo_\oneR+\unit\\
		n &\mapsto\begin{cases}
			\star,& (n=0)\\
			\pred n,& (n\neq 0).
		\end{cases}		
	\end{align*}
	There is no induction needed to verify that $s$ is a section of $r$.
\end{proof}

\begin{lemma}\label{lem:deceq}
	$\yo_\oneR$ has decidable equality in $\ShR$.
\end{lemma}
\begin{proof}
	It suffices to decide $n=_{\yo_\oneR}0$.
	We inductively define (cf. \Cref{lem:vcindbool})
	\begin{align*}
		\piop:\yo_\oneR&\to\bool\\
		0&\mapsto 0\\
		\syns  n&\mapsto 1
	\end{align*}
	and
	\begin{align*}
		\varphi:\yo_\oneR\to\bool&\to\yo_\oneR\\
		\varphi_n(0)&:=0\\
		\varphi_n(1)&:=n.
	\end{align*}
	Decidability of $\piop(n)=_\bool0$ and the identities
	\begin{align*}
		\varphi_n(\piop 0)=0\quad\quad\text{and}\quad\quad\varphi_n(\piop n)=n
	\end{align*}
	can be used to decide $n=_{\yo_\oneR}0$.
\end{proof}

\begin{theorem}\label{thm:prvbartoprv}
	There is a map
	\[\Phi:\prvbar\to\prv\]
	preserving underlying types.
\end{theorem}
\begin{proof}
	Let $(X,r,s,p):\prvbar$, $g:X$ and $h:\yo_\oneR\to X\to X$. We inductively (\Cref{thm:pruyo1}) define a helper function $\Tilde{f}:\yo_\oneR\to\yo_\oneR$ with initial value
	\[\Tilde{g}:=sg\]
	and step function
	\begin{align*}
		\Tilde h&:\yo_\oneR\to\yo_\oneR\to\yo_\oneR\\
		\Tilde h(n,x)&:=
		\begin{cases}
			sh(n,rx), &(rx:X)\\
			sg, &(rx=\star).
		\end{cases}
	\end{align*}
	Then, we define
	\begin{align*}
		f&:\yo_\oneR\to X\\
		f(n)&:=
		\begin{cases}
			r\Tilde fn, &(r\Tilde fn :X)\\
			g, &(r\Tilde fn=\star).
		\end{cases}
	\end{align*}
	We verify the computation rules of this $f$. We clearly have that $f(0)=rsg=g$. It is also true that
	\[f(\syns n)=r(\Tilde f(\syns n))=rsh(n,r\Tilde fn)=h(n,fn),\]
	if $r(\Tilde f(\syns n)):X$. Otherwise, this equality does not necessarily hold, so we need to show that $(n:\yo_\oneR)\to (x:X)\times (r\Tilde f(\syns n)=x)$. This $x$ is of course given by $h(n,r(\mop{\Tilde f} n))$, but we don't have a suitable induction principle to show this. Instead, we prove that there is a lift
  	\[
		\begin{tikzcd}[column sep=large]
			& X
			\arrow[d, "\inl"]\\
			\yo_\oneR
			\arrow[r, swap, "r\circ\Tilde f"]
			\arrow[ru, dashed]
			& X+\unit.
		\end{tikzcd}
	\]
	Since
	\[
		\begin{tikzcd}[column sep=huge]
		X
		\arrow[r, "!"]
		\arrow[d, swap, "\inl"]
		& \unit
		\arrow[d, "\inl"]\\
		X+\unit
		\arrow[r, swap, "{\langle !_X,!_{\unit}\rangle}"]
		& \unit + \unit
	\end{tikzcd}
	\]
	is a pullback square, it suffices to show that for any $n:Z\to\yo_\oneR$ there is a lift
  	\[
		\begin{tikzcd}
			& \unit
			\arrow[d, "\inl"]\\
			Z
			\arrow[r, swap,"{\langle !_X,!_{\unit}\rangle r\Tilde f n}"]
			\arrow[ru, dashed]
			& \unit+\unit.
		\end{tikzcd}
	\]
	Switching back to type theoretic language, we have to show that
	\begin{equation*}
		\langle !_X,!_{\unit}\rangle r(\mop{\Tilde f} n)=\inl\star.
	\end{equation*}
	If $n=0$, then
	\[\langle !_X,!_{\unit}\rangle r(\mop{\Tilde f} 0)=\langle !_X,!_{\unit}\rangle rsg=\langle !_X,!_{\unit}\rangle g=\inl\star.\]
	Otherwise, $n=\syns(\pred n)$, and we compute
	\begin{align*}
		\langle !_X,!_{\unit}\rangle r\Tilde f(\syns(\pred n))&=\langle !_X,!_{\unit}\rangle r\Tilde h(n,\Tilde f(\syns(\pred n))).
	\end{align*}
	There we need to make a case distinction for whether
	\[r\Tilde f(\syns(\pred n)):X.\]
	However, in both cases the right hand side reduces to $\inl\star$.
\end{proof}
\Cref{thm:prvbartoprv} gives us the non-dependent elimination principle for $\yo_\oneR$. The dependent one for a family of types in $\prubar$ can be deduced by taking the $\Sigma$-type, as we explain below.

\subsection{Closure Under \texorpdfstring{$\Sigma$}{Σ}- and Identity Types}
In this section we show that the interpretation of $\Sigma$- and identity types is sound.
\begin{lemma}
	The universe $\prvbar$ is closed under $\Sigma$-types.
\end{lemma}
\begin{proof}
	Assume we have $\Bar{A}=(A,r,s,p):\prvbar$, and $\Bar{B}(a)=(B(a),r_a,s_a,p_a):\prvbar$ given $a:A$. Then
	\begin{align*}
		\yo_\oneR\times\yo_\oneR&\to(a:A)\times B(a)\\
		(m,n)&\mapsto (\mop r m, \mop{r_{\mop r m}}n)
	\end{align*}
	has right inverse the map
	\begin{align*}
		(a:A)\times B(a)&\to\yo_\oneR\times\yo_\oneR\\
		(a,b)&\mapsto (\mop s a, \mop{s_a}b).
	\end{align*}
	We have a composite retraction \[\yo_\oneR\to\yo_\oneR\times\yo_\oneR\to (a:A)\times B(a).\qedhere\]
\end{proof}

\begin{lemma}\label{lem:prvbaridclosure}
	$\prvbar$ is closed under identity types of terms.
\end{lemma}
\begin{proof}
	Assume we have $\Bar{X}=(X,r,s,p):\prubar$ and $x,y:X$.
	We shall define a retraction $r':\yo_\oneR\to (x=_Xy)+\unit$ with section $s'$.
	By application of $r$, the type $(x=_Xy)+\unit$ is a retract of $(s(x)=_{\yo_\oneR}s(y))+\unit$.
	If $s(x)=_{\yo_\oneR}s(y)$, then $(s(x)=_{\yo_\oneR}s(y))+\unit\simeq\bool$. Otherwise, $(s(x)=_{\yo_\oneR}s(y))+\unit\simeq\unit$.
	In both cases we have a retraction $r'$ of the desired type.
\end{proof}
Since $(x=_Xy)$ denotes strict equality in $\ShR$, identity induction and its computation rules hold as well.

\section{Canonicity}\label{sec:canonicity}
In this section we use synthetic Tait computability to prove canonicity for $\prtt$. We remark that this result does not immediately follow from \cite[§ 4.5.3]{sterling_first_2021}, because the type theory presented there does not contain a type of natural numbers. Other than that, the strategy of the proof is exactly analogous in that we lift the syntax from a strong $\modal$-modal universe to a larger, strong one. Since liftings of `negative' types such as $\Pi$-, $\Sigma$- and identity types is trivial, we shall not repeat their definitions here.
\begin{theorem}\label{thm:prttcan}
	Let $m:\mr 1\to\tm 0(\synN):\prtt$ be a closed term of natural numbers type; then there exists an $n:\setN$ such that $m=\syns^n\synz$, meant as a statement about global elements in $\Set$.
\end{theorem}
\begin{proof}
	Form the glue topos $\glue$ along the left exact global sections functor $\Gamma:\Pr\prtt\to\Set$. Using the natural numbers object $\setN$ of  $\glue$, we lift $\synN, \synz, \syns$ and $\synind$. The computation rules for $\synind^*$ follow from the ones for $\synind$. The lifted terms are presented in \Cref{fig:canonicity}.
	\begin{figure*}
		\caption{Lifted Syntax for Canonicity}
		\label{fig:canonicity}
		\begin{minipage}{\fulltextwidth}
      \footnotesize
			\begin{align*}
				\synN^* &: \{ \mr{tp}^*\mid\xi\hookrightarrow \synN\}\\
				\synN^* &:= (\synN, \Sigma_{\mr{syn} : \tm 0\synN}\notmodal(\Sigma_{n : \setN} (\mr{syn} = \syns^n\synz)))\\
				\synz^* &: \{\mr{tm}^*(\synN^*)\mid\xi\hookrightarrow(\synz,\star)\}\\
				\synz^* &:= (\synz, \eta(0, \star))\\
				\syns^* &: \{\mr{tm}^*(\synN^*)\to\mr{tm}^*(\synN^*)\mid\xi\hookrightarrow(\syns, \lambdaop \_ . \star)\}\\
				\syns^* (\mr{syn}, \mr{sem}) &: \{ \Sigma_{\mr{syn} : \tm 0\synN}\notmodal(\Sigma_{n : \setN} (\mr{syn} = \syns^n\synz))\mid\xi\hookrightarrow(\syns(\mr{syn}),\star)\}\\
				\syns^* (\mr{syn}, \mr{sem}) &:= \mr{try}\: \mr{sem}\: [\alpha\mid\xi\hookrightarrow(\syns(\mr{syn}),\star)]\\
				&\qquad\text{where}\\
				&\qquad\alpha :\Sigma_{n : \setN} (\mr{syn} = \syns^n\synz) \\
				&\qquad\quad\to
				\{\Sigma_{\mr{syn}':\tm 0\synN}\notmodal\Sigma_{n : \setN} (\mr{syn}' = \syns^{n+1}\synz)\mid\xi\hookrightarrow(\syns(\mr{syn}),\star)\}\\
				&\qquad\alpha (n,p) \:= (\syns(\mr{syn}),\eta(n + 1, \mr{app}_{\syns}(p)))\\
				\synind^* &:\{(C:\synN^*\to\mr{tp}^*)\\
				&\quad\to (c_{\synz}:C(\synz))\\
				&\quad\to(c_{\syns}:(n:\synN^*)\to C(n)\to C(\syns n))\\
				&\quad\to(\mr{syn}:\synN^*)\\
				&\quad\to C(\mr{syn}))\mid\xi\hookrightarrow\synind\}\\
				\synind^*(C,c_{\synz},c_{\syns},x)&:=\mr{try}\: \mr{syn}.\mr{sem} [\beta\mid\xi\hookrightarrow\synind(C,c_{\synz},c_{\syns},\mr{syn})]\\
				&\qquad\text{where}\\
				&\qquad\beta:\Sigma_{n:\setN}(\mr{syn}=\syns^n\synz)\to \{C(\mr{syn})\mid\xi\hookrightarrow\synind(C,c_{\synz},c_{\syns},\mr{syn})\}\\
				&\qquad\beta:=\begin{cases}
					(0,\_)&\mapsto c_{\synz}\\
					(m+1,\_)&\mapsto c_{\syns}(\syns^m\synz,\beta(m,\mr{refl}))
				\end{cases}
			\end{align*}
		\end{minipage}
	\end{figure*}
	
	We remark that the partial elements $(\syns\mr{syn}, \star)$ appearing in the extent types and $\mr{try}$ statements have an implicit $\lambdaop z : \xi$ in front. The definition of $\alpha$ is valid, because
	\[z : \xi\vdash \eta_{\backslash\xi}(n + 1, \mr{app}_{\syns}(p)) = \star:\notmodal(\ldots).\]
	
	A term $m : \unit_{\Pr\prtt}\to\tm 0\synN$ lifts to $m^* : \{\unit_{\glue}\to\mr{tm}^*(\synN^*)\mid \xi\hookrightarrow m\}$. Unfolding the definition, we have a global element of $\notmodal(\Sigma_{n : \setN} (m = \syns^n\synz))$. This computes as follows.
	\begin{align*}
		&\mr{Hom}_{\glue}(1_{\glue},\notmodal(\Sigma_{n : \setN} (m = \syns^n\synz)))\\
		=&\mr{Hom}_{\glue}(1_{\glue},i_*i^*(\Sigma_{n : \setN} (m = \syns^n\synz)))\tag{by definition}\\
		\cong&\mr{Hom}_{\Set}(1_{\Set},i^*(\Sigma_{n : \setN} (m = \syns^n\synz)))\tag{$i_*$ lex}\\
		\cong&\mr{Hom}_{\Set}(1_{\Set},\Sigma_{n : \setN} i^*(m = \syns^n\synz))\tag{$i^*$ cocont.}\\
		\cong&\mr{Hom}_{\Set}(1_{\Set},\Sigma_{n : \setN} (i^*(m) = i^*(\syns^n\synz)))\tag{$i^*$ lex}\\
		=&\mr{Hom}_{\Set}(1_{\Set},\Sigma_{n : \setN} (m = \syns^n\synz))
	\end{align*}
	The last step is true, because $m$ above is actually $j_*(\yo_m))$ and so $i^*(j_*(\yo_m))=\Gammaop \yo_m=m$.
\end{proof}

\section{Soundness for Primitive Recursion}\label{sec:soundness}
In this section we use a synthetic gluing argument to show that the set-theoretical interpretation of functions actually coincides with the primitive recursive one.
The interpretations $\llb{-}_{\Set}$ and $\llb{-}_{\ShR}$ are LCC functors. Hence, they extend to left exact functors $\widehat{\llb{-}_{\ShR}}$ and $\widehat{\llb{-}_{\Set}}$ along their Yoneda embedding. It follows that we have a left exact extension
\[\rho:=\Gamma(\widehat{\llb{-}_{\ShR}})\times\widehat{\llb{-}_{\Set}}:\Pr\prtt\to\Set\]
and the comma category $\glue:=\Set\downarrow\rho$ is a topos.
\begin{definition}
	For the universe $\synUfib_0:\dot{\synU}_0\to\synU_0$ in $\prtt$ we define the morphism
	\[\llb{\synUfib_0}_{\glue}:\llb{\dot\synU_0}_{\glue}\to\llb{\synU_0}_{\glue}\]
	in $\glue$ by the square
	\[
	\begin{tikzcd}
		\Gamma(\dot\synU_0)
		\arrow[r, "\Gamma(\synUfib)"]
		\arrow[d, swap, "\llb{\dot\synU_0}_{\glue}"]
		& \Gamma(\synU_0)
		\arrow[d, "\llb{\synU_0}_{\glue}"]\\
		\rho(\dot\synU_0)
		\arrow[r, swap, "\rho(\synUfib)"]
		& \rho(\synU_0),
	\end{tikzcd}
	\]
	where the vertical arrows are each given by the `diagonal'
	\[x\mapsto(\Gamma(\widehat{\llb{x}_{\ShR}}),\widehat{\llb{x}_{\Set}}).\]
\end{definition}

\begin{lemma}
	The map $\llb{\synUfib_0}_{\glue}$ is a universe for objects $\Gamma(X)\to\rho(X)$ given by the same diagonal as above.
\end{lemma}
\begin{proof}
	A pullback square
	\[
	\begin{tikzcd}
		\dot X
		\arrow[r, ""]
		\arrow[d, ""]
		& \dot\synU_0
		\arrow[d, "\synUfib"]\\
		X
		\arrow[r, ""]
		& \synU_0,
	\end{tikzcd}
	\]
	induces a square
	\[
	\begin{tikzcd}
		\Gamma(\dot X)
		\arrow[r, ""]
		\arrow[d, ""]
		& \Gamma(\synUfib)
		\arrow[d, "\llb{\synUfib_0}_{\glue}"]\\
		\rho(X)
		\arrow[r, ""]
		& \rho(\synU_0)
	\end{tikzcd}
	\]
	in $\glue$. Because $\rho$ and $\Gamma$ preserve pullbacks, the top and bottom sides of the corresponding cube of sets are pullbacks, too. It follows that the square is a pullback in $\glue$ as well. Realignment follows from realignment of the universes in $\ShR$ and $\Set$.
\end{proof}

\begin{theorem}
	There is a sound interpretation $\llb{-}_{\glue}:\prtt\to\glue$.
\end{theorem}
\begin{proof}
	The interpretation of the type and term judgements needs to be defined first. In~\cite[Section 4.4]{sterling_first_2021}, Sterling defines a general way of lifting a syntactic model, that is an object of $\modal\llb{\synU_0}_{\glue}$ to a computability model along a hierarchy of strong universes $\llb{\synU_0}_{\glue}\leq\mc V\leq\mc W$. The larger universes are assumed to be closed under $\Pi$-types.
	
	We use the same lifting of type and term judgements and refrain from repeating it here, to keep the focus on the recursion principle.
	
	The interpretation $\llb{\synU_0}_{\glue}$ of the universe $\synU_0$ is defined in the previous lemma. It is clearly aligned over $\synU_0$. It inherits its closure under $\Sigma$- and identity types. Since $\synU_0$ is not closed under $\Pi$-types, it is not an issue that $\Gamma$ might not preserve them.
	
	We put
	\begin{align*}
		\llb{\synN}_{\glue}:\Gamma(\synN)&\to\rho(\synN)
	\end{align*}
	to be the same `diagonal' as above, aligned over $\synN$.
	
	By soundness of the interpretations $\llb{-}_{\ShR}$ and $\llb{-}_{\Set}$, any syntactic term $f:\synN\to X$ yields a filler
	\[
	\begin{tikzcd}
		\Gamma(\synN)
		\arrow[d, swap, "\llb{\synN}_{\glue}"]
		\arrow[r, "\Gamma(f)"]
		& \Gamma(X)
		\arrow[d, "{\llb{X}_{\glue}}"]\\
		\rho(\synN)
		\arrow[r, swap, "\rho(f)"]
		& \rho(X).
	\end{tikzcd}
	\]
	This shows that the obvious interpretations of $\syns$ and non-dependent $\synind$ are sound as well. The dependent elimination principle for $\synN$ follows from the non-dependent one by eliminating into and projecting from the total space of the type family.
	
	The hierarchy of universes and $\Pi$-types can be lifted in the same way as described in~\cite{sterling_first_2021}.
\end{proof}

\begin{theorem}\label{thm:prttsound}
	All $\prtt$-definable functions between the natural numbers are primitive recursive. To be precise, for any term
	\[n:\synN\vdash f(n):\synN\]
	in $\prtt$ the function $\llb{\lambdaop (n:N).f(n)}_{\Set}$ is primitive recursive.
\end{theorem}
\begin{proof}
	We made sure that the model $\llb{-}_{\glue}$ is aligned over $\prtt$. In other words, we have defined a structure-preserving section of the projection $\glue\to\PSh\prtt$.
	Hence, the assumed closed term is interpreted into the computability algebra as a global element
	\[f^*:\mr 1_{\glue}\to\{\mr{tm}^*(\synN\to\synN)^*)\mid\xi\hookrightarrow \lambdaop (n:N).f(n)\}.\]

	Unfolding the definitions, we get a commutative square
	\[
	\begin{tikzcd}[column sep=huge]
		\Gamma(\synN)
		\arrow[d, swap, "\llb{\synN}_{\glue}"]
		\arrow[r, "\Gamma(f)"]
		& \Gamma(\synN)
		\arrow[d, "\llb{\synN}_{\glue}"]\\
		\Gamma(\yo_\oneR)\times\setN
		\arrow[r, swap,"\Gamma(\widehat{\llb{f}_{\ShR}})\times\widehat{\llb{f}_{\Set}}", outer sep=4pt]
		& \Gamma(\yo_\oneR)\times\setN.
	\end{tikzcd}
	\]
	
	By canonicity, we have that $\Gamma(\synN)\cong\setN$. It is also true that $\Gamma(\yo_\oneR)\cong\setN$.	
	This implies commutativity of the diagram

  \[
		\begin{tikzcd}[column sep=large]
			\setN
			\arrow[r, "\llbracket f \rrbracket_{\Set}"]
			\arrow[d, swap, "\rotatesim"]
			& \setN
			\arrow[d, leftarrow, "\rotatesim", outer sep=4pt, pos=0.7]\\
			\Gamma(\yo_\oneR)
			\arrow[r, "\Gamma(\llb{f}_{\ShR})"']
			& \Gamma(\yo_\oneR).
		\end{tikzcd}
	\]

	Here $\Gamma(\llb{f}_{\ShR})$ is primitive recursive.
	In other words, the set-theoretic interpretation of $f$ is in fact a primitive recursive function.
\end{proof}
\begin{remark}
	Usually, when proving theorems about all open terms, such as normalisation, one has to use a \emph{figure shape} to determine which judgements are contexts. In the previous theorem we were able to avoid a more complex gluing situation by interpreting the $\lambdaop$-abstraction of $f$ and converting the global section of the exponential back to a morphism. 
\end{remark}

\section{Related Work}\label{sec:related}
The novel contribution of this work is the conservativity of primitive recursion with respect to dependent types with a full proof of soundness and canonicity.

Previously, Herbelin and Patey sketched a similar system~\cite{herbelin-patey-cprc}, the Calculus of Primitive Recursive Constructions. It is a subsystem of the Calculus of Inductive Constructions which is also conservative over Primitive Recursive Arithmetic.
The proposed soundness proof is similar in spirit, encoding types as primitive recursively decidable predicates on $\setN$. However, this simple soundness proof doesn't work with extensions to types in $\synU_1$, where we have function types, which are crucial for expressivity (even if they cannot be the target of inductions). Our system leverages the full power of dependent type theory, and is also closer to the syntax used in proof assistants such as Agda.

An extension of MLTT by additional recursion operators for well-founded relations has been studied by Paulson (c.f. \cite{PAULSON1986325}).
In recursion theory one often considers partial recursive functions. These have been studied in the context of type theory as well.
In \cite{bove_general_recursion,bove_nestedgeneralrecursion,bove_capretta_2005,bove_recursive_higher} the authors use an inductive domain predicate that characterises the inputs on which a function terminates.
Alternative approaches such as \cite{Bove2007ComputationBP,Capretta2005GeneralRV} associate to each data type a coinductive type of partial elements.
Variants of these systems have been designed for use in proof assistants (c.f. \cite{Constable1987PartialOI}).

Our system $\prtt$ is an extension of simply typed lambda calculus to full dependent type theory, potentially enhanced using a comonadic modality. One important intermediate system is the modal lambda calculus defined in \cite{pfenning_staged} in order to give a formal account of \emph{staged computation} and it is a precursor to \cite{hofmann_mixed_1998}.

\section{Conclusion and Future Work}\label{sec:conclusion}
We have shown that constructions in depedendent type theory that use elimination out of the natural numbers into universes without $\Pi$-types are primitive recursive in nature. We discuss how the system can be further conservatively extended.

A benefit of our semantic approach is that extensions can be modularly added
as long as they can be interpreted in $\ShR$ and $\glue$.
The only drawback is that such interpretations are rather involved and use a mixture of
dependent type theory (in the internal language) mixed with external reasoning about sheaves.

We expect that finitary inductive types such as lists, as well as finitary inductive type families and even small finitary induction-recursion, can be conservatively added to the calculus. This will greatly facilitate the practical mechanization of metatheory.

One could add another primitive recursive universe $\hat\synU_0:\tp 0$ with $\tm 0\hat\synU_0=\tm 0\synU_0$, yielding a code $\mr{u}_0:\tm 0(\synU_0)$ for all types $\tp 0$. One does not run into Girard's paradox due to the lack of function types.
For this, however, we need to go beyond retracts of $\yo_\oneR$, to something like $\Sigma^0_1$-definable types in $\ShR$.
A weaker version of this primitive recursive universe of codes is already definable internally by using a primitive recursive G\"{o}del encoding of the codes $\tm 0\synU_0$. However, its usefulness is uncertain because its codes only give $\tp 0$-types up to equivalence rather than judgemental equality.

It might be possible to overcome the inconvenience of not being able to define certain functions out of the natural numbers in a natural way by adding a comonadic modality $\Box$ to the theory, for example using the framework MTT~\cite{gratzer_multimodal_2021} with mode theory the walking adjunction or the walking comonad. The new type $\Box \synN$ will have a general induction principle and it should be possible to define terms  $m,n:\synN\vdash \hat f(m,n):\synN$
given $m:\Box\synN,n:\synN\vdash f(m,n):\synN$ and $m:\synN,n:\Box\synN\vdash g(m,n):\synN$ defined
by simultaneous induction together with an extensionality proof $(m, n : \synN)\vdash f(\eta m,n)=g(m,\eta n)$. 

A semantics playing the role of $\llb{-}_{\ShR}$, which we found most likely to be suitable, is as follows.
Let $\mr S$ denote the category of arities and functions between powers of natural numbers.
There is a category $\mr S\rtimes\catR$ with objects pairs of natural numbers and morphisms $(\overrightarrow u, \overrightarrow v):(\mr S\rtimes \catR)((m,n),(m',n'))$ consisting of $m'$ set-theoretic functions $u_i:\setN^m\to\setN$ and $n'$ set-theoretic functions $v_j:\setN^{m+n}\to\setN$ such that for each $x:\setN^m$ the slice $v_j(x,-)$ is primitive recursive. The projection $\pi:(\mr S\rtimes\catR)\to \mr S$ has a right adjoint inclusion functor $(-,0)$, which is also left adjoint to the functor $+:(\overrightarrow u, \overrightarrow v) \mapsto (\overrightarrow{uv},\cdot)$. The induced string of adjunctions on presheaf categories gives the required dependent right adjoint comonadic modality $+^*\circ(-,0)^*$ (c.f.~\cite[Lemma 8.2, Theorem 7.1]{gratzer_multimodal_2021}). It is yet to be checked whether $\synN$ can be soundly interpreted as $\yo_{(0,1)}$ in this setting, or if changes to the base categories, sheafification or passing to MATT~\cite{shulman_semantics_2023} are needed.

The category $\mr S\rtimes\catR$ is reminiscent of the categorical semantics of Hofmann's calculus $\mr{BC}^\omega$~\cite{hofmann_application_1997} of polytime functions with safe recursion given by the combinator
\[\mr{saferec}:\Box\mr W\to\mr W\to (\Box\mr W\to\mr W\to\mr W)\to\mr W\]
on binary strings $W$, also used in a gluing argument ot show soundness w.r.t. polytime functions.
In fact, his construction inspired ours for primitive recursion and we expect that the $\mr{BC}^\omega$ can be extended to a dependently typed system using MTT (or MATT) in a similar fashion.

Further variants of $\mr{BC}^\omega$ have been developed (c.f. \cite{hofmann_mixed_1998,hofmann_safe_2000}). However, these critically hinge on linear type systems. Hofmann's linear recursive calculus for polytime programming has been extended to a dependent type theory (see \cite{atkey2023polynomial}) using an extension of quantitative type theory (c.f. \cite{atkey2018quantitative}). It is unclear how such a system might be extended to homotopy types, see, e.g.,
\cite[Sec.~1.7.4]{riley_bunched_2022} for a discussion for quantitative type theory. Variants of linear homotopy type theory have been developed \cite{riley_bunched_2022, schreiber2014quantization}, but there is no general framework such as MTT which admits intensional identity types, linear type formers, dependent types and homotopical interpretations. The closest approximation is \cite{Licata2017AFF}, but the syntax is complicated and it does not support dependent types yet. Therefore, we do not believe that a univalent type theory for complexity classes using substructural type formers is currently within reach.

An obvious question is whether the Primitive Recursive Dependent Type Theory can be extended to a variant of Homotopy Type Theory (HoTT)~\cite{hottbook}. In plain HoTT with univalence added as an axiom there is not much hope since it is not computational. Currently, our best hope is via Cubical Type Theory (CTT)~\cite{BCH2014,CCHM2018}. Its abstract syntax is already formally defined in \cite{sterling_first_2021} so it is easy to make the necessary syntactic adjustments to its first universe. However, it is not clear how the soundness proof should be adapted because CTT cannot be interpreted in $\Set$. One could for example require that the structure maps of a cubical set be primitive recursive and embed $\Set$ into the resulting category. Even then, retracts of $\yo_\oneR$ in a category of cubical sets over $\catR$ remain of homotopy level zero, so $\llb{\synU_0}_{\ShR}$ would require a fundamentally different definition.
Another approach would be to define the cubical model in $\prtt$ itself.

\balance 
\printbibliography
\end{document}